\documentclass[10pt]{amsart}
\usepackage[english]{babel}
\usepackage{amssymb}
\usepackage{amsmath}
\usepackage{srcltx}
\usepackage{lscape}

\newcommand{\ignore}[1]{}

\newtheorem{dummy}{Dummy}

\newtheorem{lemma}[dummy]{Lemma}
\newtheorem{theorem}[dummy]{Theorem}
\newtheorem{proposition}[dummy]{Proposition}
\newtheorem{corollary}[dummy]{Corollary}

\theoremstyle{definition}

\newtheorem{example}[dummy]{Example}

\newtheorem{remark}[dummy]{Remark}



\author{S. Pumpl\"un}
\email{susanne.pumpluen@nottingham.ac.uk}
\address{School of Mathematical Sciences\\
University of Nottingham\\
University Park\\
Nottingham NG7 2RD\\
United Kingdom
}

\begin{document}

\title[Twisted algebras]
{Albert's twisted field construction using division algebras with a multiplicative norm}
\maketitle

\begin{abstract}
We generalize Albert's twisted field construction, applying it to unital division algebras with a multiplicative norm.
 We give conditions for the resulting algebras to be division algebras.
  Four- and eight-dimensional real unital and non-unital division
algebras with large derivation algebras are constructed out of Hamilton's quaternion and Cayley's octonion algebra.
\end{abstract}

\section*{Introduction}

Albert's classical construction resulting in twisted  semifields \cite{A1, A2}
 is a well-known tool to build $n$-dimensional unital nonassociative division algebras out of an $n$-dimensional cyclic
 field extension $K/F$  with Galois group
  ${\rm Gal}(K/F)=\langle \sigma\rangle$: For any choice of $c\in K^\times$ such that $N_{K/F}(c)\not=1$ and $1\leq i,j< n$,
   $K$ equipped with  the new multiplication
  $$x\circ y=xy-c\sigma^i(x)\sigma^j(y),$$
 is a division algebra over $F$. For finite base fields $F$, Kaplanski's trick is then used to associate to any such
 presemifield $(K,\circ)$ an isotopic semifield. Any semifield isotopic to $(K,\circ)$ is called a {\it twisted field}.
If $n$ is prime and $q$ large enough, any division algebra of dimension $n$ over
$\mathbb{F}_q$ is either a field or a twisted field \cite{M}. On the other hand, this construction indeed produces fewer than
$o$ non-isotopic semifields of order $o$ \cite{K}.

  In particular,
 the commutative twisted fields play a prominent role in the theory of semifields: For $c=-1$, the division algebra obtained from a
 finite field extension $K/F$ of odd degree, $F$ of odd characteristic, which is
 given by the multiplication $x\circ y=xy+\alpha(x)\alpha^{-1}(y)$ for a non-trivial $\alpha\in {\rm Gal}(K/F)$,
 $\alpha^{-1}\not=\alpha$, is a presemifield.
 Although this $(K,\circ)$ itself is not commutative,
its isotope $x\star y=x\circ \alpha(y)=x\alpha(y)+\alpha(x)y$ is, thus every associated unital isotope
 obtained by using Kaplanski's trick is
 a commutative twisted semifield. All commutative twisted semifields are obtained this way.
This construction works for every base field of odd characteristic  which
admits a cyclic field extension of degree
$n$ and yields a unital central commutative division algebra for all odd $n>2$.

Generalizations of twisted fields also occur in the classification of
three-dimensional nonassociative algebras over a field $F$ \cite{DG}.

The purpose of this paper is to show that Albert's construction can be applied to finite-dimensional algebras over a field $F$ that possess a multiplicative norm $N:A\to F$ of degree $n$, i.e. $N(xy)=N(x)N(y)$ for all $x,y\in A$, and their isotopes. As possible application, we obtain examples of real division algebras with a large derivation algebra.

We focus on unital algebras $A$ with a multiplicative norm.
 We choose $c\in A$ and two  similarities $f,g$ of $N$ to define a new multiplication on $A$.
  Since $A$ need not be commutative or associative, there are several options for a possible generalization of Albert's approach:
The possibilities include defining a new multiplication on $A$ as
 $$x\circ y=xy-c(f(x)g(y)) \text{ or } x\circ y=xy-(cf(x))g(y),$$
but we can also define a multiplicative structure by putting $c$ in the middle, i.e.
 $$x\circ y=xy-(f(x)c)g(y) \text{ or } x\circ y=xy-f(x)(cg(y)),$$
 or on the right-hand side:
$$x\circ y=xy-(f(x)g(y))c \text{ or } x\circ y=xy-f(x)(g(y)c).$$
We can moreover choose to swap around the factors in the second part of the equation and define
 $$x\circ y=xy-c(f(y)g(x))$$
 and so on, exhausting all possible combinations of this type. If $A$ is a division algebra, the right choice
 of $0\not=c\in A$ yields a division algebra $(A,\circ)$.

  We then apply Kaplanski's trick to $(A,\circ)$ to obtain a unital algebra.
  Different choices of elements $d\in A$ used for Kaplanski's trick yield isotopic unital division algebras $(A,*_d)$.
  In the theory of semifields algebras are usually classified only up to isotopy, so the choice of $d$ is not relevant in that setting.
  Choosing the unit element of $A$ when applying Kaplanski's trick to unital algebras helps to compute the inverses
   of the left and right multiplication, so that we can write down the multiplication $*$ explicitly.

After the preliminaries in Section 1, Section 2 explains how to generalize Albert's approach to (isotopes of) algebras with a multiplicative norm and gives criteria when
 the unital {\it twisted algebras} $(A,*)$ are division algebras. We construct examples of unital division algebras $(A,*)$ of dimension
 $n^2>4$ over $F$ which contain a commutative subalgebra of odd degree $n$ 
 provided the field $F$ has odd characteristic and permits a cyclic division algebra of odd degree $n$.

In Section 3 we obtain some first results on the automorphism groups and derivation algebras of some of the $(A,*)$ .
Although our main interest is to understand the twisted algebras, we will then discuss the automorphisms and derivations of some of
 the algebras $(A,\circ)$  in Section 4.
Examples of real division algebras $(A,\circ)$ and $(A,*)$, most of them with large derivation algebras, are given in
Corollary \ref{cor:real}, Examples \ref{ex:1} and \ref{ex:2}.

The problem with the construction probably becomes most apparent in the explicitly computed examples of Section 3:
Although it is rather straightforward to check whether a given twisted algebra is a division algebra or not, it is not easy to
check whether it may or may not be isomorphic to already known algebras due to its usually rather complicated multiplicative
structure.
We leave this more in depth investigation to another paper.

\section{Preliminaries}

Let $F$ be a field.

\subsection{Nonassociative algebras}

 By an ``$F$-algebra'' we mean a finite-dimensional nonassociative algebra over $F$.
 A nonassociative algebra $A\not=0$ is called a {\it division algebra} if for any $a\in A$, $a\not=0$,
the left multiplication  with $a$, $L_a(x)=ax$,  and the right multiplication with $a$, $R_a(x)=xa$, are bijective.
$A$ is a division algebra if and only if $A$ has no zero divisors.

For an $F$-algebra $A$, commutativity is measured by the {\it commutator} $[x, y] = xy - yx$ and
associativity is measured by the {\it associator} $[x, y, z] = (xy) z - x (yz)$.
The {\it left nucleus} of $A$ is defined as ${\rm Nuc}_l(A) = \{ x \in A \, \vert \, [x, A, A]  = 0 \}$, the
{\it middle nucleus} of $A$ is
defined as ${\rm Nuc}_m(A) = \{ x \in A \, \vert \, [A, x, A]  = 0 \}$ and  the {\it right nucleus} of $A$ is
defined as ${\rm Nuc}_r(A) = \{ x \in A \, \vert \, [A,A, x]  = 0 \}$.
Their intersection ${\rm Nuc}(A) = \{ x \in A \, \vert \, [x, A, A] = [A, x, A] = [A,A, x] = 0 \}$ is
 the {\it nucleus} of $A$. The nucleus is an associative
subalgebra of $A$ and $x(yz) = (xy) z$ whenever one of the elements $x, y, z$ is in
${\rm Nuc}(A)$.
 An anti-automorphism $\sigma:A\to A$ of period 2 is called an {\it involution} on $A$.
 Since $A$ is a finite-dimensional algebra over $F$, the Lie algebra of the automorphism group ${\rm Aut}(A)$, viewed as algebraic group,
 is a subalgebra of the derivation algebra ${\rm Der}(A)$ and for $F=\mathbb{R}$, we have
${\rm dim}\,{\rm Aut}(A)={\rm dim}\,{\rm Der}(A)$.

\subsection{Isotopes and Kaplanski's trick}

 Following the notation introduced in \cite[Section 1]{P}, denote the set of possibly non-unital algebra structures on an
 $F$-vector space $V$
by ${\rm Alg}(V)$. Given $A\in {\rm Alg}(V)$, we write $xAy$ for the product of $x,y\in V$ in the algebra in this subsection.

For $f,g,h\in {\rm Gl}(V)$ define the algebra $A^{(f,g,h)}$, called an
{\it isotope} of $A$, as $V$ together with the new multiplication
$$xA^{(f,g,h)}y=h(f(x)A g(y)) \quad\quad x,y\in V.$$
$A^{(f,g,h)}$ is a division algebra if $A$ is a division algebra. Two algebras $A, A'\in {\rm Alg}(V)$ are called {\it isotopic} if
$xAy=h(f(x) A' g(y))$ for all $ x,y\in V.$ If $f=g=h^{-1}$ then $A\cong A'$.

For $h=id$, we call $A^{(f,g)}=A^{(f,g,h)}$ a {\it principal Albert isotope} of $A$.
Division algebras are principal Albert isotopes of unital division algebras \cite[1.5]{P}.

Let $L_a:A\to A, x\mapsto ax$ denote the left multiplication with a non-zero $a\in A$, and $R_a:A\to A, x\mapsto xa$
 the right
 multiplication.
 Fix  non-zero $a,b\in A$. The unital division algebras isotopic to a division algebra $A$ then are, up to isomorphism,
 the algebras $A'$ given by
the new multiplication
$$x A' y=(R_a^{-1}x) A (L_b^{-1}y)$$
for all $x,y\in A$. $A'$ is a division algebra with identity element $aAb$ and is an isotope of
$A$ \cite[Proposition 9]{M}.
If we choose $a=b$, this construction is called {\it Kaplanski's trick}.
Kaplanski's trick applied to an algebra $A$ yields a unital algebra which is an Albert isotope of $A$.
Two unital algebras obtained from $A$ by choosing two elements $a,a'\in A$ and using Kaplanski's trick
with $a$ and $a'$, respectively, are isotopic to each other.

\subsection{Composition algebras} A quadratic form $N_A \colon A \to F$ on an algebra $A$ is {\it multiplicative}  if
$N_A(uv)=N_A(u)N_A(v)$ for all $u,v\in A$.
An algebra $A$ is called a {\it composition algebra} over $F$
if it admits a multiplicative quadratic form $N_A \colon A \to F$. The form $N_A$ is unique  \cite[p.~454 ff.]{KMRT}.
 It is called the {\it norm} of $A$ and we will often just write $N=N_A$.
A unital composition algebra is called a {\it Hurwitz algebra}. Hurwitz algebras are quadratic alternative and $N(1_A)=1$;
the norm of a Hurwitz algebra $C$ is the unique nondegenerate quadratic form
 on $A$ that is multiplicative.
 Hurwitz algebras
exist only in dimensions 1, 2, 4 or 8. Those of dimension~2 are exactly the
quadratic \'etale $F$-algebras, those of dimension 4 exactly the well-known
 quaternion algebras. The ones of dimension 8 are called {\it octonion algebras}. The  scalar involution
 $\overline{x} = T_{C}(x)1_A - x$ of a Hurwitz algebra $C$
 is called the {\it standard involution} of $C$, where $T_A \colon A \to F$, $T_{A}(x) = N_{A}
 (1_A, x)$, is the {\it trace} of $A$.

Every composition algebra is a principal Albert isotope of a Hurwitz algebra: There are isometries
$\varphi_1,\varphi_2$ of
the norm $N_C$ for a suitable Hurwitz algebra $C$ over $F$ such that its multiplication can be written as
$x\star y=\varphi_1(x)C\varphi_2(y)$
\cite{KMRT}.

\section{How to obtain division algebras out of a given division algebra with multiplicative norm using Albert's approach}

Unless stated otherwise, let $A$ be a
finite-dimensional nonassociative division algebra over $F$,  not necessarily unital, which possesses a nondegenerate multiplicative norm
 $N:A\to F$ of degree $n$. Let $O(N)$ denote the group of
 isometries and $S(N)$ the group of similarities of $N$.
 We choose a non-zero $c\in A$ and some $h_i,h,f,g\in S(N)$ with similarity factors
 $\gamma_i,\gamma,\alpha,\beta\in F^\times$, $1\leq i\leq 3$, respectively.

\subsection{The general construction}

Take the isotope $(A,\cdot)=A^{(h_1,h_2,h_3)}$ of $A$, then $N(x \cdot y)=\gamma_1\gamma_2\gamma_3N(x)
N(y)$ for all $x,y\in A$. If $N$ is anisotropic then it is straightforward to see that $(A,\cdot)$ is a division algebra.

 Generalizing Albert's approach, we have many options to define  new multiplications $\circ$ on $(A,\cdot)$, for instance via
 $$x\circ y=x\cdot y-c\cdot h(f(x)\cdot g(y)),$$
 $$  x\circ y=x\cdot y-h((c\cdot f(x))\cdot g(y)),$$
 $$x\circ y=x\cdot y-h(f(x)\cdot c)\cdot g(y),$$
 $$   x\circ y=x\cdot y-h(f(x)\cdot (c\cdot g(y))),$$
  $$x\circ y=x\cdot y-h((f(x)\cdot g(y))\cdot c),$$
  $$  x\circ y=x\cdot y-h(f(x)\cdot (g(y)\cdot c))\text{ etc. }$$
for all $x,y\in A$.
For noncommutative algebras, we can also swap around the factors $f(x)$ and $g(y)$ in the second part of the equation.
 Albert's proof applied to this more general setup yields the following result:

\begin{theorem} \label{thm:main}
Let $A$ be an algebra over $F$ with an anisotropic multiplicative norm $N$ of degree $n$
and  $(A,\cdot)=A^{(h_1,h_2,h_3)}$. If
$$N(c)\not= \frac{1}{\alpha\beta \gamma \gamma_1 \gamma_2 \gamma_3},$$
then $(A,\circ)$ is a division algebra over $F$.
\end{theorem}

\begin{proof}
 Since $N$ is anisotropic, the isotope $(A,\cdot)$ is a division algebra.
Let $x,y\in A$ be non-zero. Suppose that $x\circ y=0$ and rearrange the resulting equation so that one side
consists of the term $xy$.
Apply the norm on both sides.  Cancelling the
non-zero term $N(x)N(y)$, we obtain an equation that contradicts our assumption. This argument applies to all multiplications $\circ$ of the types
indicated above. For instance, if $x\circ y=x \cdot y-c \cdot h(f(x)\cdot g(y))$ then this
 yields $\gamma_1 \gamma_2 \gamma_3 N(x)N(y)=  \gamma_1 \gamma_2 \gamma_3 N(c) \gamma \gamma_1 \gamma_2 \gamma_3 N(f(x))N(g(y))$,
 thus $N(x)N(y)= \gamma \gamma_1 \gamma_2 \gamma_3\alpha \beta N(c) N(x)N(y)$.
\end{proof}

Considering isotopes of algebras with a multiplicative norm in the construction does not necessarily yield a larger class of algebras, if we only study them up to isotopy, as the following example shows:

\begin{example}
 Let $A$ be a unital algebra over $F$ with an anisotropic multiplicative norm $N$ of degree $n$
and  $(A,\cdot)=A^{(h_1,h_2,h_3)}$.
For instance, consider $x\circ y=x \cdot y-c \cdot(f(x)\cdot g(y))$ then
$$x\circ y=x \cdot y-c \cdot(f(x)\cdot g(y))=h_3(h_1(x)h_2(y))-c \cdot    (h_3(      h_1(f(x))  h_2(g(y))  ))$$
$$=h_3(h_1(x)h_2(y))-h_3(h_1(c)h_2( (h_3(  h_1(f(x))  h_2(g(y))  ))).$$
Now note that
$$h_3^{-1}(h_1^{-1}(x)\circ h_2^{-1}(y))=x y-h_1(c) h_2(  h_3( h_1(f(h_1^{-1}(x))))  h_2(g( h_2^{-1}(y) ))  ))$$
$$= xy-d h_2((h_3(h_1(f(h_1^{-1}(x)))h_2(g( h_2^{-1}(y) )))), $$
 therefore here $(A,\circ)$ is isotopic to an algebra $(A,\diamond)$ with multiplication
$x\diamond y= xy-d H( F(x)G(y) ),$
where $d\in A$ and $H,F,G\in S(N)$ are suitably chosen.
\end{example}

\subsection{The construction starting with a unital algebra}

We will only look at the possible multiplications which occur when we apply the construction to a unital algebra $A$ with a multiplicative norm. Let  $F$ be a field of characteristic 0 or $> d$. Then the
 unital division algebras with multiplicative norms are exactly the Hurwitz division
algebras and the central simple associative division algebras over  separable field extensions $K$ of $F$ \cite{Sch}.
All have anisotropic nondegenerate norms.

From now on let
$$A  \text{ be a unital division algebra over }  F  \text{ with  multiplicative norm } N \text{ of degree }n.$$
The multiplication of $A$ will be denoted by juxtaposition.
 Let $f,g\in S(N)$ with similarity factors $\alpha,\beta\in F^\times$, respectively.  Theorem \ref{thm:main} yields:

\begin{corollary}  \label{cor:important}
 If
 $N(c)\not= \frac{1}{\alpha\beta},$
  then $(A,\circ)$ is a division algebra over $F$.
  \end{corollary}

 \begin{theorem} \label{thm:main3} (cf. \cite[p.~85]{M})
 Let $K$ be a cyclic field extension of $F$ of degree $n$ with norm $N_K$ and ${\rm Gal}(K/F)=\langle \sigma\rangle.$
  For $c\in K^\times$, define
  $$x\circ y= xy-c\sigma^s(x)\sigma^t(y), \quad 0\leq s,t\leq n-1.$$
  If $s$ or $t$ is prime to $n$ and $(K,\circ)$ is a division algebra then $N_K(c)\not=1$.
\end{theorem}

From Theorems \ref{thm:main} and \ref{thm:main3} we obtain:

\begin{theorem}  \label{thm:important}
Let $K/F$ be a cyclic field extension  of degree $n$  with ${\rm Gal}(K/F)=\langle \sigma\rangle$,
which is a subalgebra of $A$. Let $c\in K^\times$. Suppose that
$f|_{K},g|_{K}\in S(N_{K/F})$  and
 $f|_{K}(x)= a\sigma^s(x)$, $g|_{K}(x)= b\sigma^t(x)$, where $a,b\in K^\times$, $0\leq s,t\leq n-1,$
 and $s$ or $t$ is prime to $n$.
 Then
$$(A,\circ) \text{ is a division algebra if and only if } N_{K/F}(abc)\not=1.$$
  \end{theorem}

\begin{proof}
 Obviously, here $\alpha=N(a)$ and $\beta=N(b)$.
 By Corollary \ref{cor:important}, $(A,\circ)$ is a division algebra  if $N(c)\not=1/N(a)N(b)$, i.e. $N_{K/F}(abc)\not=1$.
 \\ Conversely, suppose that $(A,\circ)$ is a division algebra.
 Since $c\in K^\times$, $(K,\circ)$ is a subalgebra of $(A,\circ)$ with multiplication
$x\circ y=xy-abc \sigma^s(x) \sigma^t(y).$ By assumption, $s$ or $t$ is prime to $n$, and $(K,\circ)$ is a division algebra, so by Theorem
 \ref{thm:main3} 
it follows that $N_{K/F}(abc)\not=1$.
\end{proof}

 \begin{corollary}
 Let $A$ be a quaternion or octonion division algebra over $F$. Let  $c\in A$ such that $K=F(c)$ is a separable quadratic
field extension with non-trivial automorphism $\sigma$.  Suppose that
 $f|_{K}(x)= a\sigma^s(x)$, $g|_{K}(x)= b\sigma^t(x)$ for all $x\in K$, where $a,b\in K^\times$, and $s=1$ or $t=1$. Then
$(A,\circ)$ is a division algebra if and only if $N_{K/F}(abc)\not=1$.
 \end{corollary}

\begin{corollary}\label{cor:cyclic}
Let $K/F$ be a Galois field extension of degree $n\geq 2$ with ${\rm Gal}(K/F)=\langle \sigma\rangle$ and
$A=(K/F,\sigma,d)$ be a cyclic division algebra over $F$ of degree $n$. Let $c\in K^\times$ and $f|_K,g|_K\in S(N_{K/F})$.
Let $f|_{K}(x)= a\sigma^s(x)$, $g|_{K}(x)= b\sigma^t(x)$, where $a,b\in K^\times$, and $s$, $t$ be integers, $0\leq s,t\leq n-1$, with $s$ or $t$ prime to $n$. Then $(A,\circ)$ is a division algebra if and only if $N_{K/F}(abc)\not=1$.
\end{corollary}

To obtain a unital algebra $(A,*)$ from $(A, \circ)$  we apply Kaplanski's trick.
Different choices of $d\in A$ used for Kaplanski's trick yield isotopic unital division algebras $(A,*_d)$.
  In the theory of semifields this is not relevant as division algebras over finite fields are usually classified up to isotopy.
Let $L_x$ and $R_y$ be the left and right multiplication
on $(A,\circ)$ and let $e=1_A$ be the unit element of $A$. Define a  multiplication using Kaplanski's trick via
$$x*y=R_e^{-1}(x)\circ L_e^{-1}(y)$$
for all $x,y\in A$. We call the unital algebra $(A,*)$  a {\it twisted algebra}. By construction,
$(A,*)$ is isotopic to $(A,\circ)$, and is a division algebra if and only if so is $(A,\circ)$.
Note that choosing the unit element of $A$ when applying Kaplanski's trick to unital algebras helps to compute the inverses
   of the left and right multiplication for when we want to write down the multiplication $*$ explicitly.

We note that Menichetti gives a geometric condition for when a division algebra of dimension $n$ is isomorphic to a twisted algebra,
involving its left and right zero divisor hypersurfaces
\cite[Corollary 32]{M}.

\begin{lemma}
Let  $B$ be a subalgebra of $A$ with norm $N_B$ and $0\not=c\in {B}$.
Assume $f|_{B},g|_{B}\in S(N_{B})$. Then $(B,*)$ is a unital subalgebra of $(A,*)$.
\end{lemma}

\begin{proof}
Let $x,y\in B$. Since $A$ is unital, clearly $e=1_A\in B$ and $(B,\circ)$ is a subalgebra of $(A,\circ)$ by construction.
We know that $x\circ e,$ $e\circ y\in B$, and thus the restricted maps
$R_e:B\to B$, $L_e:B\to B$ are isomorphisms onto $B$. We conclude that $x*y=R_e^{-1}(x)\circ L_e^{-1}(y)\in B$ as well.
\end{proof}

 \begin{corollary}\label{prop:subtwistedfield0}
 Let $A$ be a quaternion or octonion division algebra over $F$. Choose $c\in A$ such that $K=F(c)$ is a separable quadratic
field extension with  automorphism $\sigma$.  Suppose that
 $f|_{K}(x)= a\sigma^s(x)$, $g|_{K}(x)= b\sigma^t(x)$ for all $x\in K$, where $a,b\in K^\times$ and $0\leq s,t\leq 1$.
 Then
$(A,*)$ is a division algebra if and only if $N_{K/F}(abc)\not=1$. \\
 In particular, $(A,*)$ contains the two-dimensional  subalgebra
$(K,*)$.
 \end{corollary}

For $n>2$, unital division algebras of dimension $n^2$ containing isotopes of commutative
twisted algebras of dimension $n$ as subalgebras can be constructed out of cyclic division algebras as follows:

\begin{proposition}\label{prop:subtwistedfield}
Let $K/F$ be a Galois field extension of degree $n\geq 2$ with ${\rm Gal}(K/F)=\langle \sigma\rangle$ and
$A=(K/F,\sigma,d)$ be a cyclic division algebra over $F$ of degree $n$. Let $c\in K^\times$ and $f|_K,g|_K\in S(N_{K/F})$.
 Let $f|_{K}(x)= a\sigma^s(x)$, $g|_{K}(x)= b\sigma^t(x)$, where $a,b\in K$, and $s$, $t$ are integers, $0\leq s,t\leq n-1$, with $s$ or $t$ prime to $n$. Then $(A,*)$ is a unital division algebra if and only if $N_{K/F}(abc)\not=1$. \\
If additionally $s\not=t$, $s\not=0$, $t\not=0$, then $(A,*)$ has the twisted algebra
$(K,*)$ as a subalgebra. 
\\  In particular, suppose $n$ is odd, $s+t=n$, $s\not=0$, $t\not=0$ and $abc=-1$. Then
 the subalgebra $(K,*)$ of $(A,*)$ is isotopic to an $n$-dimensional commutative twisted algebra.
\end{proposition}

\begin{proof}
 For all $c\in A$ such that $N_{K/F}(abc)\not=1$,
the twisted algebra $(A,*)$ 
is a division algebra by Theorem \ref{thm:main}.
 Since   $c\in K^\times$ and $f|_K,g|_K\in S(N_K)$, $(A,\circ)$ contains the $n$-dimensional subalgebra $(K,\circ)$, and
   $(A,*)$ contains the $n$-dimensional unital subalgebra $(K,*)$.
 Moreover, $(A,\circ)$ and thus
 $(A,*)$ are division algebras if and only if $N(abc)\not=1$ (Theorem \ref{thm:important}).
 If additionally $s\not=t$, $s\not=0$, $t\not=0$, then the subalgebra 
$(K,*)$ is a twisted division algebra.

If $n$ is odd, $s+t=n$, $s\not=0$, $t\not=0$ and $abc=-1$,
then $(K,\circ)$, 
$x\circ y=xy+\sigma^s(x) \sigma^{-s}(y)$, is a subalgebra of $(A,\circ)$.
The isotope $(K,\star)$ given by
$x \star y=x\circ \sigma^s(y)=x\sigma^s (y)+\sigma^s(x) y$ of $(K,\circ)$ is a commutative algebra, and Kaplanski's trick applied to $(K,\star)$ yields a unital $n$-dimensional commutative division algebra.
\end{proof}

More precisely, we now define $\circ_i$ on $A$ by
\begin{enumerate}
\item $x\circ_{(1)} y=xy-c(f(x)g(y)),$
\item $  x\circ_{(2)} y=xy-(cf(x))g(y),$
\item $x\circ_{(3)} y=xy-(f(x)c)g(y),$
\item $   x\circ_{(4)} y=xy-f(x)(cg(y)),$
 \item $x\circ_{(5)} y=xy-(f(x)g(y))c,$
 \item $  x\circ_{(6)} y=xy-f(x)(g(y)c)$
\end{enumerate}
for all $x,y\in A$.
For noncommutative algebras $A$, we can swap around the factors in the second part of the equation and  define
\begin{enumerate}
\setcounter{enumi}{6}
\item $x\circ_{(7)} y=xy-c(f(y)g(x)),$
\item $x\circ_{(8)} y=xy-(cf(y))g(x),$
\item $x\circ_{(9)} y=xy-(f(y)c)g(x),$
\item $  x\circ_{(10)} y=xy-f(y)(cg(x)),$
\item $x\circ_{(11)} y=xy-(f(y)g(x))c,$
\item $ x\circ_{(12)} y=xy-f(y)(g(x)c)$
\end{enumerate}
for all $x,y\in A$. 
Note that if $B$ is a subalgebra of $A$, $f|_{B},g|_{B}\in S(N_{B})$ and $c\in {B}^\times$ then $(B,\circ)$ is a subalgebra of $(A,\circ)$.

If the twisted algebra $(A,*)$ is obtained from $(A,\circ_{(i)})$, we denote it by $(A,*_{(i)})$.
We still write $(A,*)=(A,*_{(i)})$, $i=1,\dots,12$, for the unital algebra obtained by applying Kaplanski's trick employing the unit element of $A$, if it is clear from the context which multiplication $*_{(i)}$ we use.

Note that all the possible multiplications $\circ_{(i)}$ in Proposition \ref{prop:subtwistedfield} are given by
$i=1,3,5,7,9,11$
since $A$ is associative here. All of them become the same multiplication on the subalgebra $(K,\circ)$.

\begin{remark}
 The algebras $(A,\circ_{(i)})$ we construct can be isotopic or even isomorphic to  $A$:
 Suppose that $f=g\in {\rm Aut}(A)$. Put $G(x)=x-f(x) c$. Then
$$  x\circ_{(1)} y=xy-c(f(x)f(y))=
(id -cf)(xy) \text{ and  } x\circ_{(5)} y=xy-f(xy)c=G(xy),$$
and for all $c\in A$ such that $N(c)\not=1$,  $(A,\circ_{(1)})$ and $(A,\circ_{(5)})$ are  division algebras by Theorem \ref{thm:main}.
 However, $(A,\circ_{(1)})$ and $(A,\circ_{(5)})$ are  isotopic to $A$, since
 the maps $F=id- cf$ and $G(x)=x-f(x) c$ are bijective if  $N(c)\not=1$.  Hence if $A$ is associative, we obtain
$(A,*_{(1)})\cong (A,*_{(5)})\cong A$. 
The same argument applies when $A$ is a quaternion algebra and $f=g=\sigma$ is the canonical involution of $A$. Again, $(A,*_{(1)})\cong (A,*_{(5)})\cong A$.
\end{remark}

The question when two division algebras we construct are not isotopic will be dealt with in another paper.

\section{Some observations on the automorphisms and derivations of the unital division algebras $(A,*)$}

Let $A$ be a unital division algebra with (anisotropic) multiplicative norm $N$, $0\not=c\in A$  and $f, g\in S(N)$.
We look at a few special cases where we can say something on the automorphism group and the derivation algebra of $(A,*)$.

\subsection{Central simple algebras}

Let $A$ be a central simple division algebra over $F$ of degree $n$.

\begin{remark} \label{ex:isom}
The multiplications
$$ x\circ_{(1)} y=xy-cf(x)y=(id-cf)(x)y,$$
$$ x\circ_{(1)} y=xy-cxg(y)=x(id-cg)(y)
\text{ for } c\in F^\times,$$
$$ x\circ_{(1)} y=xy-cf(x)f(y)=(id-cf)(xy) \text{ if } f\in {\rm Aut}(A),$$
$$  x\circ_{(5)} y=xy-xg(y)c=x(y-g(y)c),$$
  $$  x\circ_{(5)} y=xy-f(x)f(y)c=T(xy)  \text{ if }  f\in {\rm Aut}(A),  \text{ with }  T(z)=z-f(z)c,$$
$$  x\circ_{(7)} y=xy-c f(y)f(x)=(id-cf)(xy)  \text{ if }  f(xy)=f(y)f(x), $$
$$    x\circ_{(11)} y=xy-f(y)f(x)c=xy-f(xy)c=S(xy)  \text{ if } f(xy)=f(y)f(x) \text{ with } S(z)=z-f(z)c,$$
all yield isotopes of $A$, provided the  maps $id-cf$, $T$, $S$ etc. are bijective, which they are when we choose $c\in A$ suitable
for $(A,\circ)$ to be division algebras.
Therefore in these  cases, the twisted unital algebras $(A,*)$
 are isotopic to the unital associative algebra $A$, thus have the same nucleus, and therefore are
 isomorphic to $A$.
\end{remark}

Suppose $f,g\in {\rm Aut}(A)$, $f\not=id$, $g\not=id$.

Let $\tau$ be an involution on $A$ and $c\in F^\times$, $c\not=\pm 1$, and
$$(7.1)\quad x\circ_{(7.1)} y=xy-c\tau(y)x,$$
$$(7.2)\quad x\circ_{(7.2)} y=xy-cy\tau(x),$$
$$(7.3)\quad x\circ_{(7.3)} y=xy-c\tau(x)\tau(y).$$
Then we can explicitly compute
$$(7.1)\quad x*_{(7.1)}y
= \frac{1}{(1-c)(1-c\tau(c))}( xy-c\tau(y)x+cx\tau(y)-c^2yx), $$
$$(7.2)\quad x*_{(7.2)}y
=\frac{1}{(1-c)(1-c\tau(c))}(xy-cy\tau(x)+c\tau(x)y-c^2 yx ),$$
$$(7.3)\quad x*_{(7.3)}y
= (1-c\tau(c))^{-2} ( (1-c\tau(c)^{2})xy-c(1-c)\tau(x)\tau(y)+c(1-\tau(c))(x\tau(y)+\tau(x)y)  ). $$

An easy calculation shows:

\begin{proposition}\label{prop:invol}
For the multiplications $*_{(7.i)}$, $i\in\{1,2,3\}$, we obtain:
\\ (i) If $H\in {\it Aut}(A)$ such that $H\circ\tau=\tau\circ H$, then $H\in {\rm Aut}(A,*_{(7.i)})$.
\\ (ii) If $D\in {\rm Der}(A)$ such that $D(\tau(x))=\tau(D(x))$ for all $x\in A$, then
$D\in {\rm Der}(A,*_{(7.i)})$.
\end{proposition}

\begin{corollary}
 Suppose $A$ contains a quaternion subalgebra $C$ with canonical involution $\sigma$ and assume that
$\tau|_C=\sigma$. For the multiplications $*_{(7.i)}$, $i\in\{1,2,3\}$, we obtain:
\\ (i) $\{H\in {\rm Aut}(A)\,|\, H(x)=dxd^{-1}, \, d\in C\}\subset {\rm Aut}(A,*_{(7.i)})$ and
${\rm Aut}(A,*_{(7.i)})$
contains a subgroup isomorphic to $SU(2)$.
\\ (ii) $\{D\in {\rm Der}(A)\,|\, D(x)=ax-xa, \, a\in C\}\subset {\rm Der}(A,*_{(7.i)})$
and ${\rm Der}(A,*_{(7.i)})$ contains a subalgebra isomorphic to $su(2)$ .
\end{corollary}

\begin{proof}
(i) We know that
$H(x)=dxd^{-1}$ for some $d\in A$, therefore $H(\tau(x))=d\tau(x)d^{-1}$,
$\tau(H(x))=\tau(dxd^{-1})=\tau(d^{-1})\tau(x)\tau(d) $, so
$H(\tau(x))=\tau(H(x))$ if and only if $d\tau(x)d^{-1}=\tau(d^{-1})\tau(x)\tau(d)$ if and only if
$\tau(x)=d^{-1}\tau(d^{-1})\tau(x)\tau(d)d$  if and only if
$x\tau(d)d=\tau(d)dx$ for all $x\in A$ if and only if $\tau(d)d\in {\rm Comm}(A)=F$.
 Now  $\tau(d)d=\sigma(d)d=N_C(d)\in F^\times$ for all $d\in C^\times$.
 The canonical embedding
$$\{H\in {\rm Aut}(C)\,|\, H(x)=dxd^{-1}, \, d\in C\}\hookrightarrow \{H\in {\rm Aut}(A,*_{(7.i)})\,|\, H(x)=dxd^{-1}\}, \, H\mapsto H$$
 implies that $SU(2)\cong {\rm Aut}(C)$ is isomorphic to a subgroup of ${\rm Aut}(A,*_{(7.i)})$.
 \\ (ii) We have $D(x)=ax-xa$ for some non-zero $a\in A$, therefore $D(\tau(x))=\tau(D(x))$ if and only if
 $(a+\tau(a))\tau(x)=\tau(x)(a+\tau(a))$ for all $x\in A$, which is equivalent to $a+\tau(a)\in {\rm Comm}(A)=F$.
 Now $a+\tau(a)=a+\sigma(a)$ for all $d\in D^\times$.
The canonical embedding
$$\{D\in {\rm Der}(C)\,|\, D(x)=ax-xa, \, a\in C\}\hookrightarrow \{D\in {\rm Der}(A,*_{(7.i)})\,|\, D(x)=ax-xa\}, \, D \mapsto D$$
implies that $su(2)$ is isomorphic to a subalgebra of ${\rm Der}(A,*_{(7.i)})$.
\end{proof}

 For $c\in A$, define
$${\rm Aut}_c(A)=\{f\in {\rm Aut}(A)\,|\, f(c)=c\}
\text{ and } {\rm Der}_c(A)=\{D\in {\rm Der}(A)\,|\, D(c)=0\}.$$
If $c\in F^\times$   then ${\rm Aut}_c(A)={\rm Aut}(A)$ and ${\rm Der}_c(A)={\rm Der}(A)$.

For our next setup suppose now that $f,g\in {\rm Aut}(A)$ with $f^m=id$, $g^m=id$ for some $m\geq2$  and $c\in A$ such that
$1\not=cf(c)f^2(c)\cdots f^{m-1}(c)$, $1\not=cg(c)g^2(c)\cdots g^{m-1}(c)$ are chosen to define $\circ$.
Note that if $F$ does not contain a primitive $m$th root of unity, then the conditions on $c$ are trivially satisfied, if $c\in F$.
The proof of the next lemma is a simple computation, once we have computed the multiplication $*$ explicitly (which is straightforward but tedious):

\begin{proposition}\label{prop:aut}
 (i) If $H\in {\rm Aut}_c(A)$ such that $H\circ f=f\circ H$, $H\circ g=g\circ H$,  then $H\in {\rm Aut}(A,*)$.
\\ (ii) If $D\in {\rm Der}_c(A)$ such that $D(f(x))=f(D(x))$ and $D(g(x))=g(D(x))$ for all $x\in A$, then
$D\in {\rm Der}(A,*)$.
\end{proposition}

If the multiplication $\circ$ is defined instead by using an involution $\tau$ of $A$ and an automorphism $f$, such that $f^m=id$ for some $m\geq2$ and
$1\not=cf(c)f^2(c)\cdots f^{m-1}(c)$, or using
two different involutions of $A$, the multiplication $(A,*)$ can be explicitly computed analogously and
 results corresponding to Proposition \ref{prop:aut} and Proposition \ref{prop:invol} hold.

\subsection{Hurwitz algebras}

Let $A$ be a quaternion or octonion division algebra over $F$. Recall that  $h\in {\rm Aut}(A)$ is a \emph{reflection} of $A$, if $h^2=id$.
 Let $f,g\in {\rm Aut}(A)$ be two reflections of $A$.

\begin{proposition}\label{prop:19}
Let $c\in A$ such that $N(c)\not=1$.
 \\ (i) Let $f\not=id$, $g\not=id$, $c\in F^\times$ and $ x \circ_{(1)}y=xy-cf(x)g(y)$, then $(A,*_{(1)})$ is a division algebra.
\\
(ii) Let $A$ be a quaternion division algebra  and $f\not=id$, $g\not=id$, $cf(c)\not=1$, $cg(c)\not=1$, and
$  x\circ_{(1)} y=xy-cf(x)g(y)$, $ x\circ_{(3)} y=xy-f(x)cg(y), $ $ x\circ_{(5)} y=xy-f(x)g(y)c,$
$  x\circ_{(7)} y=xy-cg(y)f(x), $ $  x\circ_{(9)} y=xy-g(y)cf(x), $ or $    x\circ_{(11)} y=xy-g(y)f(x)c$.
then $(A,*_{i})$ is a division algebra for $i\in \{1,3,5,7,9,11\}$.
\\ (iii) Let $A$ be a quaternion division algebra, $f=id$,
 $c\in A\setminus F$ such that $cg(c)\not=1$, and
 $x\circ_{(1)} y=xy-c x g(y)$. Then $(A,*_{(1)})$ is a division algebra.
\end{proposition}

This is clear from our main result. Moreover, in all the above cases, we can compute the multiplication $*$ explicitly, sometimes under additional assumptions on $c$. Put
$u=(1-cf(c))^{-1}$, $v=(1- f(c) c )^{-1}$, $w=(1-cg(c))^{-1}$, $s=(1-g(c)c  )^{-1}\in A$.
\\
In (i), if $c^2\not=1$, then
$$ x*_{(1)}y=(1-c^2)^{-2}  (1-c^3)xy+(1+c)^{-2}c(1-c)^{-1}(  xg(y)+f(x)y+f(x)g(y)  ).$$
In (ii), if $0\not=c\in A$, $cf(c)\not=1\not =cg(c)$, then
$$ x*_{(1)}y=u (x+c f(x) ) w(y+c g(y) )
-c( u (f(x)+f(c)x )     )  s(g(y)+g(c)y )  ),$$
$$  x*_{(3)}y= v(x+ f(x)c ) w (y+c g(y))
- u (f(x)+ xf(c)  )c s (  g(y)+g(c)y )    ),$$
$$  x*_{(5)}y=
v(x+f(x)c) s (x+g(x)c)
- u (f(x)+xf(c))    ) w (g(x)+xg(c)) )c,$$
$$  x*_{(7)}y=u(x+cf(x)) w (x+cg(x))-c(  s (g(x)+g(c)x)  ) v (f(x)+f(c)x),  $$
$$  x*_{(9)}y=u(x+cf(x)) s (x+g(x)c) - w (g(x)+xg(c)) c v ( f(x)+f(c)x),   $$
$$  x*_{(11)}y=
u (x+cf(x)) s (x+g(x)c)  -(  w (g(x)+xf(c)) v (f(x)+f(c)x)  c.$$
In (iii), if  $c\in A\setminus F$ $c\not=1$, $cg(c)\not=1$, $f=id$, then
$$  x*_{(1)}y
= (1-c)^{-1}x w (y+c g(y) )-c (1-c)^{-1}x s (g(y)+g(c) y ).$$

Note that if $f=g$ in Proposition \ref{prop:19} (i), (ii), then the unital algebra $(A,*)$ is clearly an isotope of $A$.
In (i), the choice of $f=id$ or $g=id$  also yields unital algebras $(A,*)$ isotopic to $A$ (and thus isomorphic to $A$ if they are four-dimensional).

We do not explicitly compute the remaining possible cases in (ii), where we choose $id$
and a reflection in multiplications $\circ_{(i)}$ for $i\in\{1,3,5,7,9,11\}$, since they
 are analogous to the case given in (ii).
Some of these $(A,*)$ again yield algebras isomorphic to quaternion algebras. 

Suppose now that $A$ is a quaternion algebra and $0\not=c\in A$  such that $N(c)\not=1$ and
$cf(c)\not=1$,  $cg(c)\not=1$.
We look at some explicitly computed multiplications:
\\
 $(7.1)\quad x\circ_{(7.1)} y=xy-cf(y)x$, so 
$$  x*_{(7.1)}y=(1-c)^{-1}x u (y+c f(y) )-c v(f( y)+f(c) y  )(1-c)^{-1} y.$$
 $(7.2)\quad x\circ_{(7.2)} y=xy-c y g(x)$, so 
$$  x*_{(7.2)}y= w (x+c g(x) ) (1-c)^{-1}y-c (1-c)^{-1}y  s(g( x)+g(c) x  ).$$
 $(7.3)\quad x\circ_{(7.3)} y=xy-c f(y) g(x)$, so
$$  x*_{(7.3)}y=w (x+c g(x) ) u (y+c f(y) )-c  v (f( y)+f(c) y  )s(g( x)+g(c) x  ).$$
The proofs are straightforward but tedious calculations. The algebras $(A,*_{(7.1)})$, $(A,*_{(7.2)})$. $(A,*_{(7.3)})$ are  four-dimensional unital division algebras.

\begin{lemma}
 (a)
Suppose $c\in F^\times$,  $c^2\not=1$, then
any $H\in {\rm Aut}(A)$ such that $f\circ H=H\circ f$, $g\circ H=H\circ g$ lies in ${\rm Aut}(A,*_{(1)})$.
\\ (b) Let $A$ be a quaternion division algebra over $F$. Let $f,g\in {\rm Aut}(A)$, $f\not=id$, $g\not=id$
be two reflections and   $0\not=c\in A$, $cf(c)\not=1$, $cg(c)\not=1$.
Let $H\in {\rm Aut}_c (A)$.
 Then:
 \\ (i) if $H\circ f=f\circ H$, then $H\in {\rm Aut} (A,*_{(7.1)})$
 \\  (ii) if $H\circ g=g\circ H$, then $H\in {\rm Aut} (A,*_{(7.2)})$,
\\ (iii)  if $H\circ f=f\circ H$ and $H\circ g=g\circ H$, then $H\in {\rm Aut} (A, ,*_{(7.3)})$.
\end{lemma}

\begin{corollary}
 Let $A$ be a quaternion algebra and  $f\not=id$, $g\not=id$
be two reflections, $f(x)=sxs^{-1}$, $g(x)=txt^{-1}$,  and $0\not=c\in A$, $cf(c)\not=1$, $cg(c)\not=1$.
Then
$$\{H\in {\rm Aut}_c(A) \,|\, H(x)=dxd^{-1},\, d\in F(s) \}\subset \{H\in {\rm Aut}_c(A) \,|\, H\circ f=f\circ H \}\subset {\rm Aut} (A,*_{(7.1)}),$$
$$\{H\in {\rm Aut}_c(A) \,|\, H(x)=dxd^{-1},\, d\in F(t) \}\subset \{H\in {\rm Aut}_c(A) \,|\, H\circ g=g\circ H \}\subset {\rm Aut} (A,*_{(7.2)}),$$
$$\{H\in {\rm Aut}_c(A) \,|\, H(x)=dxd^{-1},\, d\in F(t)\cap F(s) \}\subset
\{H\in {\rm Aut}_c(A) \,|\, H\circ f=f\circ H, \,\,H\circ g=g\circ H \}$$
$\subset {\rm Aut} (A,*_{(7.3)}).$
\end{corollary}

Let now $\tau$ be an involution on $A$  and $c\in F^\times$ such that
$1\not=c \tau(c)$.
Then for $c\in F^\times$, we obtain
$$(7.1)\quad   x\circ_{(7.1)} y=xy-c\tau(y)x,\quad  x*_{(7.1)} y= \frac{1}{(1-c)(1-c\tau(c))}( xy-c\tau(y)x+cx\tau(y)-c^2yx), $$
$$(7.2)\quad  x\circ_{(7.2)} y=xy-c y \tau(x), \quad x*_{(7.2)}y=\frac{1}{(1-c)(1-c\tau(c))}(xy-cy\tau(x)+c\tau(x)y-c^2 yx ),$$
$$(7.3)\quad x\circ_{(7.3)} y=xy-c  \tau(x)\tau(y), $$
 $$x*_{(7.3)}y= (1-c\tau(c))^{-2} ( (1-c\tau(c)^{2})xy-c(1-c)\tau(x)\tau(y)+c(1-\tau(c))(x\tau(y)+\tau(x)y)  ).$$

\begin{proposition} \label{prop:25}
Let $A$ be a quaternion or octonion division algebra and $c\in F^\times$ such that $c\not=\pm 1$
and $c\tau(c)\not=1$. Let $(A,*)$ be as in cases $(7.1)$, $(7.2)$ or $(7.3)$.
\\ (i) If $f\in {\rm Aut}(A)$ such that $f\circ\tau=\tau\circ f$ then $f\in {\rm Aut}(A,*)$.
\\ (ii) If $\tau$ is the canonical involution of $A$ then $(A,*)$ is a division algebra, and
$$ SU(2)\cong {\rm Aut}(A)\subset  {\rm Aut}(A,*)$$
if $A$ is a quaternion algebra and
$$ G_2\cong {\rm Aut}(A)\subset  {\rm Aut}(A,*)$$
if $A$ is an octonion algebra.
\\ (iii) If $\tau$ is the canonical involution of $A$ and $A$ a quaternion algebra then
$$ su(2)\cong {\rm Der}(A)\subset  {\rm Der}(A,*).$$
\end{proposition}

\begin{proof}
(i) and (iii) are straightforward calculations, (ii)  follows from (i) and \cite[p.~85]{M}.
\end{proof}

More generally, similar considerations yield the following:

\begin{proposition}\label{prop:26}
Let $A$ be a quaternion or octonion division algebra over $F$ with canonical involution $\sigma$, $c\in A$ such that
$N(c)\not=1$ and  $f,g\in \{ id,\sigma\}$. Take any possible definition of $\circ$ using these $f$ and $g$.
If $A$ has dimension 8, we additionally assume that $c\in F^\times$. Then
$${\rm Aut}_c(A)\subset {\rm Aut}(A,*).$$
 In particular, for $c\in F^\times$,
$$ SU(2)\cong {\rm Aut}(A)\subset  {\rm Aut}(A,*)$$
if $A$ is a quaternion algebra and
$$ G_2\cong {\rm Aut}(A)\subset  {\rm Aut}(A,*)$$
if $A$ is an octonion algebra.

\end{proposition}

\begin{proof} This follows from the fact that the canonical involution commutes with any automorphism of $A$.
\end{proof}

We point out that in the above setting,  $(A,*)$ is a division algebra if and only if $N(c)\not= 1$.
Using the classification of real division algebras \cite{B-O1} we obtain:

\begin{corollary}\label{cor:real}
Let $F=\mathbb{R}$ and $A$ a quaternion or octonion division algebra over $\mathbb{R}$  with canonical involution $\sigma$.
Let  $c\in \mathbb{R}^\times$ such that $c\not=\pm 1$. Take any possible definition of $\circ$ using $f,g\in \{ id,\sigma\}$. Then
$$ su(2)\cong {\rm Der}(A,*)$$
if $A$ is a quaternion division algebra and
$$ G_2\cong  {\rm Der}(A,*)$$
if $A$ is an octonion division algebra. In particular, this is true for the multiplications  $(7.1), (7.2)$, $(7.3)$ with $\tau=\sigma$
from Proposition \ref{prop:25}.
\end{corollary}

Concerning subalgebras, since $\sigma$ restricted to any subalgebra of $A$ is the canonical involution of that subalgebra, we state:

\begin{lemma}
Let $A$ be a quaternion or octonion division algebra over $F$.
 Take any possible definition of $\circ$ using $f,g\in \{ id,\sigma\}$.
\\ (i) If $c\in F^\times$ then every subalgebra $B$ of $A$ yields a subalgebra $(B,\circ)$ of $(A,\circ)$ and
$(B,*)$ of $(A,*)$.
\\ (ii) If $c\in A\setminus F$ then every subalgebra $B$ of $A$ such that $c\in B$
 yields a subalgebra $(B,\circ)$ of $(A,\circ)$ and $(B,*)$ of $(A,*)$.
 \\ (iii) Suppose $D$ is a subalgebra of $A$. If $c\not\in D$ then $(D,\circ)$ is not a subalgebra of $(A,\circ)$.
\end{lemma}

\begin{proof}
We only show (iii), as the rest is trivial: Assume that $(D,\circ)$ is a subalgebra of $(A,\circ)$, then
$x\circ y\in D$ for all $x,y\in D$, so in particular, $1\circ 1=1-c\in D$ and thus also $c\in D$, contradiction.
\end{proof}

%
\section{Automorphisms and derivations of the non-unital algebras $(A,\circ)$}
%

Let $A$ be an algebra over $F$ with a nondegenerate multiplicative norm $N$ of degree $n$, and $0\not=c\in A$.

\begin{lemma} Let $f,g\in O(N)$.
\\ (i) Let $H\in {\rm Aut}_c(A)$
 such that $H(f(x))=f(H(x))$ and $H(g(x))=g(H(x))$ for all
$x\in A$. Then $H\in {\rm Aut}(A,\circ)$.
\\ (ii)  If $D\in {\rm Der}(A)$ such that $D(c)=0$, $D(f(x))=f(D(x))$ and $D(g(x))=g(D(x))$ for all $x\in A$, then
$D\in {\rm Der}(A,\circ)$.
\end{lemma}

\begin{proof}
(i) Consider $x\circ_{(1)} y=xy-c(f(x)g(y))$.
For $H\in {\rm Aut}(A)$, $H(x\circ_{(1)} y)=H(x)H(y)-H(c)(H(f(x))H(g(y)))$ and $H(x)H(y)=H(x)H(y)-c(f(H(x))g(H(y)))$.
So if $H(c)=c$ and $H(f(x))=f(H(x))$ as well as $H(g(x))=g(H(x))$ for all $x\in A$, $H\in {\rm Aut}(A,\circ)$.
A similar argument applies for all $\circ_{(i)}$.
\\ (ii) The proof is a simple computation.
\end{proof}

\begin{proposition} \label{cor:35}
Let $A$ be a central simple associative algebra over $F$ and $f,g\in {\rm Aut}(A)$, i.e. $f(x)=sxs^{-1}$ and
$g(x)=txt^{-1}$ where $s,t\in A^\times$.
\\ (i) If $c\in A\setminus F$ then
$$\{ d_a\in {\rm Der}(A)\,|\, a\in F(s)\cap F(t) \cap  F(c) \}\subset\{ d_a\in {\rm Der}_c(A)\,|\, a\in F(s)\cap F(t)\}$$
$$\subset \{ d_a\in {\rm Der}_c(A)\,|\, f,g\in {\rm Aut}_a(A)  \}\subset {\rm Der}(A,\circ).$$
If $c\in F^\times$ then
$$\{ d_a\in {\rm Der}(A)\,|\, a\in F(s) \cap F(t)  \}\subset
\{ d_a\in {\rm Der}(A)\,|\, f,g\in {\rm Aut}_a(A)  \}\subset {\rm Der}(A,\circ).$$
\\ (ii)  Suppose $c\in A\setminus F$.
If  $s,t\in F(c)$ or $c\in F(s)\cap F(t)$ then
$d_c\in {\rm Der}(A,\circ).$
\end{proposition}

\begin{proof}
(i)
Every derivation of $A$ has the form
$d_a(x)=ax-xa$ for all $x\in A$. Let $f(x)=sxs^{-1}$ and
$g(x)=txt^{-1}$, $s,t\in A^\times$. Then $D(f(x))=f(D(x))$ and $D(g(x))=g(D(x))$ implies that if
$f(a)=a$, $g(a)=a$, it follows that $d_a\in {\rm Der}(A,\circ)$. Now  $f(a)=a$ is the same as $sa=as$, so that this holds for all
$a\in F(s)$, and analogously, $g(a)=a$ is the same as $ta=at$, so that this holds for all
$a\in F(t)$.  If $c\in F^\times$ then $D(c)=0$. The assertions follow.
\\ (ii) We know that $0\not=d_c\in  {\rm Der}_c(A)$ by (i). Now if $s,t\in F(c)$ then $f(c)=c$ and $g(c)=c$ and so
$d_c\in {\rm Der}(A,\circ)$. Alternatively, if $c\in F(s)\cap F(t)$ then also $f(c)=c$ and $g(c)=c$.
\\ (iii) is now clear.
\end{proof}

\begin{example} \label{ex:1}
Let $F=\mathbb{R}$, $A=\mathbb{H}$   and $f,g\in S(N)$ with similarity
factors $\alpha$ and $\beta$.  For all $0\not=c\in \mathbb{H}$ such that $N(c)\not=1/\alpha\beta$,
 $(\mathbb{H},\circ)$ is a division algebra.
 Suppose that $f,g\in {\rm Aut}(\mathbb{H})$ and $f(x)=sxs^{-1}$,
$g(x)=txt^{-1}$ with $s,t\in \mathbb{H}^\times$.
By Corollary \ref{cor:35} and \cite{B-O1},
 then
 ${\rm dim}\, {\rm Der}(\mathbb{H},\circ)=1$ or ${\rm Der}(\mathbb{H},\circ)\cong su(2)$.

 If $f|_\mathbb{C},g|_\mathbb{C}\in S(N_\mathbb{C})$ and $c\in \mathbb{C}^\times$, then $(\mathbb{H},\circ)$ contains the two-dimensional subalgebra $(\mathbb{C},\circ)$.
 Choose any $d\in \mathbb{H}$ and apply Kaplanski's trick. This yields a unital
 division algebra $(\mathbb{H},*_d)$.  If $d\in \mathbb{C}^\times$,
$(\mathbb{C},*_d)$ is a unital subalgebra isomorphic to $\mathbb{C}$.
\end{example}

For $F=\mathbb{R}$, we have
${\rm dim}\,{\rm Aut}(A)={\rm dim}\,{\rm Der}(A)$

 \begin{example} \label{ex:2}
 Let $F=\mathbb{R}$ and $\mathbb{O}$ be Cayley's octonion division algebra with norm $N$.
\\ (i)
Let $c\in \mathbb{R}^\times$ and
$f,g\in \{ id,\sigma\}$. Then $(\mathbb{O},\circ)$
 is a division algebra if and only if $N(c)\not=1$. Moreover,
$G_2\cong {\rm Aut}(\mathbb{O})\subset {\rm Aut}(\mathbb{O},\circ)$ by Proposition \ref{prop:26}.
Thus
$$G_2\cong {\rm Aut}(\mathbb{O},\circ) \text{ and } G_2\cong {\rm Der}(\mathbb{O},\circ)$$
by \cite{B-O1}.
\\ (ii) Let $f,g\in S(N)$ with similarity
factors $\alpha,\beta$.
For all $0\not=c\in \mathbb{O}$ such that $N(c)\not=1/\alpha\beta$, $(\mathbb{O},\circ)$ is a division algebra.

If $f|_\mathbb{H},g|_\mathbb{H}\in S(N)$
and $c\in \mathbb{H}^\times$, $(\mathbb{O},\circ)$ contains the four-dimensional
 subalgebra $(\mathbb{H},\circ)$.
Then choose any $0\not=d\in \mathbb{O}$  to
  apply Kaplanski's trick to $(\mathbb{O},\circ)$. This yields a unital division algebra $(\mathbb{O},*_d)$.
If additionally $d\in \mathbb{H}$,
$(\mathbb{H},*_d)$ is a unital  four-dimensional subalgebra of $(\mathbb{O},*_d)$ and
contains $\mathbb{C}$ as subalgebra
if $f|_\mathbb{C},g|_\mathbb{C}\in S(N_\mathbb{C})$ and $c\in \mathbb{C}$.
 \end{example}


\end{document}